\documentclass[letterpaper, 10 pt, conference]{ieeeconf}  % Comment this line out if you need a4paper

\IEEEoverridecommandlockouts                              % This command is only needed if 
\usepackage{graphicx}
\usepackage{amsmath, amssymb, amsfonts}
\usepackage{cite}
\usepackage{mathtools}
\let\labelindent\relax
\usepackage{enumitem}
\usepackage[ruled,vlined]{algorithm2e}
\newtheorem{thm}{Theorem}% [section]
\newtheorem{lem}[thm]{Lemma}
\newtheorem{prp}[thm]{Proposition}
\newtheorem{cor}[thm]{Corollary}
\newtheorem{ass}[thm]{Assumption}

\newtheorem{aprp}{Proposition}% [section]

\DeclareMathOperator*{\tr}{tr}
\DeclareMathOperator*{\logdet}{log\,det}

\newcommand{\norm}[1]{\lVert #1 \rVert}                                 % ノルム
\newcommand{\inX}[2]{\in\mathbb{#1}^{#2}}                               % 集合の表記用
\newcommand{\sqmat}[4]{\begin{bmatrix}#1&#2\\#3&#4\end{bmatrix}}        % 2-2行列
\newcommand{\matvec}[2]{\begin{bmatrix}#1\\#2\end{bmatrix}}             % 2-1行列
\newcommand{\hDel}{\hat{\Delta}}                                        % データ摂動関連の略記（下二つも同様）
\newcommand{\st}{\ \mathrm{s.t.}\ }                                     % 

\newcommand{\el}{\mathcal{E}}
\newcommand{\rank}{\mathrm{rank}\,}
\newcommand{\diag}{\mathrm{diag}}
\newcommand{\Spr}[2]{\mathbb{S}^{#1,#2}_{++}}

\newcommand{\alphat}[1]{\alpha^{(#1)}}
\newcommand{\talphat}{\tilde{\alpha}^{(t+1)}}
\newcommand{\betat}{\beta^{(t)}}
\newcommand{\Rt}{R^{(t)}}
\newcommand{\Qt}{Q^{(t)}}
\newcommand{\Akt}{A_k^{(t)}}
\newcommand{\Bkt}{B_k^{(t)}}
\newcommand{\Ct}{C^{(t)}}

\mathtoolsset{showonlyrefs}

\title{\LARGE \bf Scalable Outer Approximation of Minkowski Sums\\ of Matrix Ellipsoids for Data-Driven Control}

\author{Taira Kaminaga$^{1}$ and Hampei Sasahara$^{2}$% <-this % stops a space
\thanks{*This work was supported in part by JST SPRING, Japan Grant Number JPMJSP2180, JSPS KAKENHI Grant Number JP24K17296, and JST-ASPIRE Program Grant Number JPMJAP2402.}% <-this % stops a space
\thanks{$^{1}$Taira Kaminaga is with Department of Systems and Control Engineering, Graduate School of Engineering,
Institute of Science Tokyo, Tokyo, Japan {\tt\small kaminaga.t.3734@m.isct.ac.jp}}%
\thanks{$^{2}$Hampei Sasahara is with Department of Information Physics and
Computing, Graduate School of Information Science and Technology, The University of Tokyo, Tokyo, 113-8656 Japan
        {\tt\small hsasahara@g.ecc.u-tokyo.ac.jp}}%
}

\begin{document}

\maketitle
\thispagestyle{empty}
\pagestyle{empty}

%%%%%%%%%%%%%%%%%%%%%%%%%%%%%%%%%%%%%%%%%%%%%%%%%%%%%%%%%%%%%%%%%%%%%%%%%%%%%%%%
\begin{abstract}
%In data-driven control, matrix ellipsoids—a matrix-valued extension of standard geometric ellipsoids—are increasingly utilized to bound unmeasured noise and disturbances. However, capturing diverse and complex prior knowledge often requires outer-approximating the Minkowski sum of multiple matrix ellipsoids with a single new matrix ellipsoid. While existing methods formulate this approximation using linear matrix inequalities (LMIs), the number of variables grows quadratically with the data length, rendering the computation intractable for standard optimization solvers. To overcome this severe computational challenge, this paper proposes a fast outer-approximation method that entirely circumvents massive LMI constraints by introducing a parameterized family of bounding matrix ellipsoids. For minimizing the sum of squared semi-axes, we derive an exact analytical solution. For volume minimization, we develop an efficient algorithm based on the majorization-minimization framework. By constructing a surrogate function via the first-order approximation of the log-determinant function, our approach provides closed-form update rules, drastically reducing the per-iteration computational complexity. Furthermore, we theoretically prove that the algorithm monotonically decreases the volume and that the parameter sequence asymptotically converges to the set of stationary points. Numerical experiments demonstrate that the proposed method exhibits vastly superior computational efficiency and scalability compared to interior-point solvers.
Matrix ellipsoids provide a standard framework for representing bounded uncertainties in data-driven control. Since noise models for sequential observations are naturally represented as the Minkowski sum of multiple matrix ellipsoids, applying existing robust control methods, which typically assume a single ellipsoidal set, requires a tight outer approximation. While techniques based on linear matrix inequalities (LMI) are applicable, their computational cost grows quadratically with the data length, limiting their scalability. This paper investigates the optimal outer approximation problem under two criteria: the sum of squared semi-axes and the volume. We propose an LMI-free approach by introducing a parameterized family of bounding matrix ellipsoids. Specifically, we derive an exact analytical solution for the first criterion and develop an efficient majorization-minimization (MM) algorithm for the second. The proposed MM algorithm employs a first-order approximation of the log-determinant function to provide closed-form update rules, ensuring monotonic convergence to the set of stationary points. Numerical experiments demonstrate that our method offers significantly higher computational efficiency and scalability than standard interior-point solvers.
\end{abstract}

%%%%%%%%%%%%%%%%%%%%%%%%%%%%%%%%%%%%%%%%%%%%%%%%%%%%%%%%%%%%%%%%%%%%%%%%%%%%%%%%
\section{INTRODUCTION}

%In control theory, ellipsoids have been widely utilized to represent bounded uncertainties. Characterized solely by a center vector and a shape matrix, ellipsoids require significantly fewer parameters to describe than other geometries like polyhedra. Leveraging this property, numerous methods have been proposed to approximate and estimate the possible bounds of future state trajectories and system parameters using ellipsoids, under the assumption of unknown-but-bounded disturbances and measurement noise~\cite{ellip:FOGEL_Automatica_1982,ellip:kurzhanski2000,ellip:durieu_multi-input_2001}.
%During prediction and estimation with error bounds, algorithms must compute Minkowski sums and intersections of ellipsoids to account for accumulating disturbances and updated measurements. Consequently, techniques to approximate the results of these set operations with a single new ellipsoid have been extensively developed, forming a crucial theoretical foundation for ellipsoid-based prediction and estimation~\cite{ellip:durieu_multi-input_2001,ellip:kurzhanskiy_TAC_2007,ellip:yan_closed-form_2015,ellip:Halder_CDC_2018}.

In control theory, ellipsoids are widely used to represent bounded uncertainties. Compared to other shapes such as polyhedra, ellipsoids are much simpler to describe, making them highly practical for capturing the possible bounds of future states and system parameters under unknown bounded disturbances and measurement noise~\cite{ellip:FOGEL_Automatica_1982,ellip:kurzhanski2000,ellip:durieu_multi-input_2001}. In real-time applications, handling sequential observations naturally involves Minkowski sums of ellipsoids to account for accumulating disturbances. Because standard algorithms typically assume a single bounding ellipsoid, the exact results of these operations must be approximated by a single ellipsoid. Consequently, techniques for such approximations have served as a crucial foundation for ellipsoid-based prediction and estimation~\cite{ellip:durieu_multi-input_2001,ellip:kurzhanskiy_TAC_2007,ellip:yan_closed-form_2015,ellip:Halder_CDC_2018}.

Extending the standard vector-valued framework, matrix ellipsoids have been introduced to represent matrix-valued uncertainties. This representation is particularly useful in data-driven control, where controllers for unknown systems are synthesized directly from noisy time-series data. The key idea in this approach is that both the bounds on the noise sequences and the desired control specifications can be formulated as matrix ellipsoids. Consequently, the matrix S-lemma~\cite{DDCQMI:Waarde2022_TAC_origin,DDCQMI:Waarde2023_siam_qmi} allows the inclusion relations between these ellipsoids to be equivalently transformed into computationally tractable linear matrix inequalities (LMIs). This powerful property has enabled the systematic derivation of LMI conditions for a wide range of control objectives under various noise settings~\cite{DDCQMI:Waarde2022_TAC_origin,DDCQMI:Waarde2023_siam_qmi ,DDCQMI:BISOFFI2022_Petersen, DDCQMI:Steentjes2022_Cont_Sys_Let_Covariance, DDCQMI:Waarde2024_TAC_AR, DDCQMI:BISOFFI2024_CSL, DDCQMI:Lidong2026, DDCQMI:Kaminaga2025_ACC, DDCQMI:kaminaga2025datainformativity}.

%To address this limitation, studies in data-driven control have individually investigated instantaneous constraints~\cite{DDCQMI:BISOFFI2021,DDCQMI:BISOFFI2024_CSL} and bounds on the maximum noise amplitude~\cite{DDCQMI:Hu2025_RDPC}. Although these studies are primarily motivated by controller design, the mathematical challenge of outer-approximating a noise set with complex constraints using a single matrix ellipsoid is also treated as a crucial step.
%Indeed, \cite{DDCQMI:kaminaga2025datainformativity} has shown that such diverse noise constraints can be comprehensively represented as the Minkowski sum of multiple matrix ellipsoids. That study derives LMI-based sufficient conditions for outer-approximating this Minkowski sum with a single matrix ellipsoid, allowing existing ellipsoid-based control methods to be applied under these diverse noise constraints.
%However, the LMI characterizing this approximation contains a number of variables proportional to the square of the data length, and its dimension grows proportionally with the data length. Consequently, the optimization problem scales rapidly, posing a severe computational issue that makes the LMI difficult to solve within a practical timeframe.

However, as in the vector-valued case, a single matrix ellipsoid is often insufficient to capture specific prior knowledge such as independent bounds on multiple noise sources or instantaneous constraints at each time step. To address this, several works in data-driven control have individually formulated bounding methods for instantaneous noise limits~\cite{DDCQMI:BISOFFI2021,DDCQMI:BISOFFI2024_CSL} and maximum noise amplitudes~\cite{DDCQMI:Hu2025_RDPC}. Recently, \cite{DDCQMI:kaminaga2025datainformativity} has unified these properties by representing the noise set as the Minkowski sum of multiple matrix ellipsoids. They have derived LMI conditions to outer-approximate this Minkowski sum with a single ellipsoid, thereby keeping the subsequent control synthesis tractable. While this LMI-based approach provides a comprehensive theoretical solution, its computational cost can be a limitation in practice, for example, in real-time applications. Because the number of matrix-valued decision variables and their dimension grow quadratically and linearly with the data length, respectively, the optimization problem scales poorly, making it difficult to solve efficiently as the dataset grows.

%To address this issue, we propose a fast outer-approximation method for Minkowski sums of matrix ellipsoids. First, we characterize a parameterized family of approximating matrix ellipsoids using only positive scalar parameters, equal in number to the ellipsoids being summed. Constructed based on the LMI sufficient conditions in~\cite{DDCQMI:kaminaga2025datainformativity}, this parameterization achieves an approximation equivalent to directly solving the original massive inequality constraints. This fundamentally avoids the rapid explosion of variables associated with conventional LMI-based characterizations, significantly improving computational efficiency. Simultaneously, this result serves as a natural extension of the vector-valued ellipsoid approximation method proposed in~\cite{ellip:durieu_multi-input_2001} to matrix-valued representations.

To address this issue, we propose a scalable outer-approximation method for Minkowski sums of matrix ellipsoids. A key insight, drawn from the classical vector-valued case~\cite{ellip:durieu_multi-input_2001}, is that the massive LMI conditions in~\cite{DDCQMI:kaminaga2025datainformativity} can be equivalently reformulated using a parameterized family of bounding ellipsoids. Crucially, this parameterization relies solely on positive scalar variables equal in number to the ellipsoids being summed. This avoids the variable explosion inherent in conventional LMI characterizations, enabling highly efficient computation without loss of approximation quality.

%Next, we determine the optimal matrix ellipsoid from this parameterized family based on criteria such as the sum of squared semi-axes or the volume. For minimizing the sum of squared semi-axes, we apply the Lagrange multiplier method from~\cite{ellip:durieu_multi-input_2001} with appropriate normalization to derive an analytical solution for the optimal parameters. 
%Since it requires only matrix multiplications and trace computations, this analytical solution is a powerful and efficient tool for accurately obtaining the approximating matrix ellipsoid.

%Conversely, for volume minimization, the objective function is non-convex, unlike in the vector-valued case. Therefore, we develop an efficient parameter update rule based on the majorization-minimization (MM) algorithm~\cite{math:sum2017MMalgo}. Specifically, exploiting the concavity of the log-determinant function, we construct a surrogate function that upper-bounds the objective and sequentially update the parameters by minimizing it. We successfully derive this update equation in closed form, drastically reducing the computational cost. Furthermore, we theoretically prove that this algorithm monotonically decreases the volume of the approximating matrix ellipsoid, and that the generated parameter sequence asymptotically converges to the set of stationary points of the objective function.

Building upon this parameterization, we determine the optimal bounding matrix ellipsoid by minimizing either the sum of squared semi-axes or the volume. For the first criterion, we extend the Lagrange multiplier method from~\cite{ellip:durieu_multi-input_2001} with appropriate normalization to derive an exact analytical solution. Requiring only fundamental matrix operations, this closed-form solution provides a highly efficient tool for accurate approximation. In contrast, the second criterion, minimizing the volume, presents a fundamental mathematical challenge. Although volume minimization is typically a convex problem in the classical vector-valued case, extending it to the matrix-valued setting renders the objective function inherently non-convex. To overcome this difficulty, we develop an efficient parameter update rule based on the majorization-minimization (MM) algorithm~\cite{math:sum2017MMalgo}. By exploiting the concavity of the log-determinant function, we construct a surrogate upper bound and derive a closed-form update equation that drastically reduces computational cost. Furthermore, we theoretically prove that this algorithm monotonically decreases the volume and asymptotically converges to the set of stationary points.
Finally, we conduct numerical experiments to verify the effectiveness of the proposed method.

%Finally, we conduct numerical experiments to verify the effectiveness of the proposed method. Because the volume minimization for vector-valued approximation reduces to a convex optimization problem, a rigorous comparison with conventional convex solvers is possible. Through these experiments, we demonstrate that the proposed volume minimization algorithm exhibits overwhelming computational speed and scalability compared to general-purpose interior-point solvers.

This paper is organized as follows.
Section~\ref{sec:pre} reviews the preliminary results.
Section~\ref{sec:prob} formulates the problem of outer approximation of the Minkowski sum of matrix ellipsoids.
In Section~\ref{sec:parametrize}, we introduce our key parameterization technique to characterize the bounding matrix ellipsoids.
Section~\ref{sec:opt} presents the methods for determining the optimal ellipsoid.
Its effectiveness is numerically demonstrated in Section~\ref{sec:numerical}, and Section~\ref{sec:conclusion} concludes the paper.

\subsection*{Notation}
We denote the set of integers by $\mathbb{Z}$,
the set $\{a, a+1, \dots, b\}$ for $a, b \in \mathbb{Z}$ with $a < b$ by $[a, b]$,
the set of $n$-dimensional vectors by $\mathbb{R}^n$,
the set of $n$-dimensional vectors with all positive and non-negative elements by $\mathbb{R}^n_{++}$ and $\mathbb{R}^n_+$ respectively,
the set of $n \times n$ symmetric matrices by $\mathbb{S}^n$,
the set of $n \times n$ positive (semi)definite matrices by $\mathbb{S}^n_{++} (\mathbb{S}^n_+)$,
the $n \times n$ identity matrix by $I_n$,
the $n \times m$ zero matrices by $0_{n,m}$,
the block diagonal matrix with blocks $A_1, A_2, \dots, A_K$ by $\mathrm{diag}_{i \in [1, K]}(A_i)$,
the Kronecker product of matrices $A$ and $B$ by $A \otimes B$,
the Minkowski sum of sets $\mathcal{S}_1$ and $\mathcal{S}_2$ by $\mathcal{S}_1 \oplus \mathcal{S}_2$,
the gradient of a function $f(x)$ by $\nabla f(x)$ or $\nabla_x f(x)$,
the positive and negative (semi)definiteness of a symmetric matrix $M$ by $M \succ (\succeq) 0$ and $M \prec (\preceq) 0$ respectively,
and the positive semidefinite square root of $A \succeq 0$ by $A^\frac{1}{2}$.
% The subscript is omitted when the dimension is clear from the context.

\section{Preliminaries}\label{sec:pre}
\subsection{Matrix Ellipsoids}\label{subsec:matellip}
This subsection explains the definition and properties of matrix ellipsoids.
First, a matrix ellipsoid defined by $C\inX{R}{q\times r}$, $Q\inX{S}{q}_{++}$, and $R\inX{S}{r}_{++}$ is formulated as
\begin{equation}\label{e:EllipDfn}
    \el(C,Q,R)\coloneqq \{X\inX{R}{q\times r}\mid(X-C)^\top Q^{-1}(X-C)\preceq R\}. %\subseteq\mathbb{R}^{p\times r}
\end{equation}
Note that this definition has a degree of freedom regarding parameter scaling; $\el(C,Q,R)=\el(C,a^{-1}Q,aR)$ holds for any $a>0$. Furthermore, setting $r=1$ and $R=1$ recovers the standard vector-valued ellipsoid as a special case.

In data-driven control, the noise matrix is often assumed to be bounded by a known matrix ellipsoid. In this context, representations using quadratic matrix inequalities have been widely used, such as
\begin{equation}\label{e:QuadForm}
    \matvec{I_r}{X}^\top \Pi\matvec{I_r}{X}\succeq 0,
\end{equation}
with which a matrix ellipsoid can be described as $\el(C,Q,R)=\{X\mid \eqref{e:QuadForm}\}$, where $\Pi$ and $(C,Q,R)$ correspond to each other by
\begin{equation}\label{e:Pi_CPR}
    \Pi=\sqmat{R-C^\top Q^{-1}C}{C^\top Q^{-1}}{Q^{-1}C}{-Q^{-1}}.
\end{equation}
For notational simplicity, we hereinafter denote $\Spr{q}{r}= \mathbb{S}^q_{++}\times \mathbb{S}^r_{++}$.
A rigorous analysis of the matrix representation \eqref{e:QuadForm} is provided in \cite{DDCQMI:Waarde2023_siam_qmi}, which yields the following useful properties.

\begin{prp}\label{prp:MatEllip}
Let $(Q,R)\in\Spr{q}{r}$.
Then, the following statements hold.
\begin{enumerate}[label=(\alph*), ref=\theprp(\alph*)]
    \item \label{prp:bound} $\el(C,Q,R)$ is bounded and has a nonempty interior.
    \item \label{prp:explicit} $\el(C,Q,R)$ can be explicitly expressed as
    \begin{equation}\label{e:Explicit}
        \el(C,Q,R)=\{C+Q^{\frac{1}{2}}SR^{\frac{1}{2}}\mid \norm{S}_2\le 1\}.
    \end{equation}
\end{enumerate}
\end{prp}

Proposition~\ref{prp:MatEllip} follows directly from \cite[Theorem 3.2]{DDCQMI:Waarde2023_siam_qmi} and \cite[Theorem 3.3]{DDCQMI:Waarde2023_siam_qmi}.
Regarding Proposition~\ref{prp:explicit}, a vector-valued ellipsoid with $R=1$ can be written as $\el(c,Q,1)=\{c+Q^{\frac{1}{2}}s\mid \norm{s}_2\le 1\}$.
In this case, $\sqrt{\det Q}$ is proportional to the volume of the ellipsoid, and $\tr Q$ represents the sum of squared semi-axes. 
To extend these geometric concepts to matrix ellipsoids, we consider the vectorization of matrix ellipsoids as in \cite[Section 2]{DDCQMI:BISOFFI2021}.
Let $\mathrm{vec}:\mathbb{R}^{q\times r}\rightarrow \mathbb{R}^{qr}$ be the vectorization operator. 
By the properties of the Kronecker product, we have
\begin{align}
    \el_{\rm v}(C,Q,R)&\coloneqq \{\mathrm{vec}(X)\mid X\in\el(C,Q,R)\}\\
    &=\{\mathrm{vec}(C+Q^{\frac{1}{2}}SR^{\frac{1}{2}})\mid \norm{S}_2\le 1\}\\
    &= \{c_{\rm v}+(R\otimes Q)^{\frac{1}{2}}\mathrm{vec}(S) \mid \norm{S}_2\le 1\},
\end{align}
where $c_{\rm v}=\mathrm{vec}(C)$. 
This indicates that the volume and the sum of squared semi-axes of a matrix ellipsoid can also be evaluated as functions of $R\otimes Q$.

In this paper, we focus on matrix ellipsoids centered at the origin, defining them as $\el^0(Q,R)\coloneqq \el(0,Q,R)$. From Proposition~\ref{prp:explicit}, any matrix ellipsoid can be written as $\el(C,Q,R)=\{C\}\oplus\el^0(Q,R)$.

\subsection{Data-Driven Control with Matrix Ellipsoids}\label{subsec:cnt}
To clarify our motivation, we describe the relationship between matrix ellipsoids and data-driven control~\cite{DDCQMI:Kaminaga2025_ACC,DDCQMI:kaminaga2025datainformativity}.
%This subsection overviews the problem setting and results from \cite{DDCQMI:Kaminaga2025_ACC,DDCQMI:kaminaga2025datainformativity}.
Consider a discrete-time linear time-invariant system $x_+=A^*x+B^*u$, where $x_+\in\mathbb{R}^n,x\in\mathbb{R}^n,$ and $u\in\mathbb{R}$ are the subsequent state, the current state, and the input, respectively.
The true system matrices $(A^*,B^*)$ are unknown, but noisy input-state data is available. 
In an ideal, noise-free environment, an offline data batch data $(X_+^*,X^*,U^*)$ of length $T$ satisfies the system dynamics
\begin{equation}\label{cnt:cleandata}
    X_{+}^*=A^*X^*+B^*U^*,
\end{equation}
where $X_{+}^*\in\mathbb{R}^{n\times T}$, $X^*\in\mathbb{R}^{n\times T}$, and $U^*\in\mathbb{R}^{m\times T}$.
The available measured data $(X_+,X,U)$ is assumed to be corrupted by additive perturbations $(\Delta_{X_+},\Delta_X,\Delta_U)$ as
\begin{equation}\label{cnt:measured_data}
    X_+=X_+^*+\Delta_{X_+},\quad X=X^*+\Delta_X,\quad U=U^*+\Delta_U,
\end{equation}
whose dimensions match those of $(X_+^*,X^*,U^*)$.
By defining the data matrix $\mathbf{X}\coloneqq[X_+^\top~ -X^\top ~ -U^\top]^\top$ and the perturbation matrix $\Delta\coloneqq[\Delta_{X_+}^\top ~ -\Delta_X^\top ~ -\Delta_U^\top]^\top$, equations \eqref{cnt:cleandata} and \eqref{cnt:measured_data} yield $[I_n~A^*~B^*](\mathbf{X}-\Delta)=0$.
Furthermore, the unknown perturbation $\Delta$ is assumed to satisfy
\begin{equation}\label{cnt:noise_matellip}
    \Delta=E\hDel,\quad \hDel^\top\in\el(C,Q,R),
%\Delta^\top \in E\el(C,Q,R)
\end{equation}
where $E\inX{R}{(2n+m)\times \hat{n}}$ and $(C,Q,R)\inX{R}{T\times \hat{n}}\times \mathbb{S}^T_{++}\times \mathbb{S}^{\hat{n}}_{++}$ are known matrices.
The goal is to design a controller $u=Kx$ based on the data $(X_+,X,U)$ and prior knowledge of perturbations in~\eqref{cnt:noise_matellip}, ensuring desired performance for the closed-loop system $x_+=(A^*+B^*K)x$.
This problem setting has been widely studied in the literature. Specifically, setting $E=[I_n~ 0_{n,n+m}]^\top$ corresponds to studies dealing with process disturbances \cite{DDCQMI:Waarde2022_TAC_origin,DDCQMI:BISOFFI2022_Petersen,DDCQMI:Steentjes2022_Cont_Sys_Let_Covariance}, while $E=I_{2n+m}$ covers studies considering measurement noise \cite{DDCQMI:BISOFFI2024_CSL,DDCQMI:Lidong2026}.

Let $\Sigma$ be the set of all systems consistent with the measured data:
\begin{equation}
    \Sigma\!\coloneqq\!\{(A,B)\,|\, \text{$\exists\Delta$ s.t. $[I_n~A~B](\mathbf{X}-\Delta)\!=0$ and \eqref{cnt:noise_matellip} hold}\}.
\end{equation}
By definition, $(A^*,B^*)\in\Sigma$ trivially holds. 
Since the designer cannot distinguish $(A^*,B^*)$ from other systems in $\Sigma$, performance guarantees for the true system are achieved by ensuring the performance for all systems in $\Sigma$.
%Practically, a controller $K$ that guarantees this performance can be obtained by solving LMIs derived via the matrix S-lemma.
%As mentioned in the introduction, LMI conditions for various control objectives have already been developed.
LMI conditions on the controller $K$ for various control objectives have been developed via the matrix S-lemma~\cite{DDCQMI:Waarde2022_TAC_origin,DDCQMI:Waarde2023_siam_qmi}.

%To design a controller using this approach, one must determine the matrices $E$, $C$, $Q$, and $R$ satisfying \eqref{cnt:noise_matellip}. However, this format of prior knowledge does not always align with the actual information available in practical applications.
However, a single matrix ellipsoid is often insufficient to capture specific prior knowledge in practice.
%To address this issue, \cite{DDCQMI:kaminaga2025datainformativity} sought to improve the representational capability of \eqref{cnt:noise_matellip}, generalize the noise model~\eqref{cnt:noise_matellip} to the following form:
To address this issue, \cite{DDCQMI:kaminaga2025datainformativity} generalizes the noise model~\eqref{cnt:noise_matellip} to the form
\begin{equation}\label{cnt:noise_minkowski}
\textstyle{
    \Delta^\top\in \bigoplus_{k=1}^KF_k \el(C_k,Q_k,R_k)G_k.
    }
\end{equation}
This class includes, e.g., (i) independent information on multiple noise sources, (ii) bounds on maximum noise amplitude, and (iii) instantaneous noise constraints.
To handle this type of noises, the work \cite{DDCQMI:kaminaga2025datainformativity} has derived LMI-based sufficient conditions for $(C,Q,R)$ to satisfy the outer approximation
\begin{equation}\label{cnt:noise_kinji}
\textstyle{
    \bigoplus_{k=1}^KF_k \el(C_k,Q_k,R_k)G_k\subseteq\el(C,Q,R),
    }
\end{equation}
and proposed a method to solve these LMIs simultaneously with those for controller design.

However, the sufficient LMI conditions for \eqref{cnt:noise_kinji} involve a massive number of variables proportional to $T^2$. 
Moreover, in cases (ii) and (iii), the number of summed ellipsoids $K$ also grows proportionally to $T$, severely compromising its scalability. 
To resolve this issue, this paper derives an algorithm to compute $(C,Q,R)$ satisfying \eqref{cnt:noise_kinji} with low computational complexity.

\section{Problem Formulation}\label{sec:prob}

We consider the problem of outer approximation for the Minkowski sum of matrix ellipsoids with a single matrix ellipsoid.
Consider $K$ matrix ellipsoids $\el_k\coloneqq \el(C_k,Q_k,R_k)$, where $(C_k,Q_k,R_k)\in \mathbb{R}^{q_k\times r_k}\times\mathbb{S}^{q_k}_{++}\times \mathbb{S}^{r_k}_{++}$ for $k\in[1,K]$.
Let the Minkowski sum of these ellipsoids under linear transformations be defined as
\begin{align}
\textstyle{
    \mathcal{T}\coloneqq \bigoplus_{k=1}^K F_k\el_kG_k,
    }
\end{align}
where $F_k\inX{R}{q\times q_k}$ and $G_k\inX{R}{r_k\times r}$ for $k\in[1,K]$ are known constant matrices.
We aim to outer-approximate this set with a single matrix ellipsoid, namely, $(C,Q,R)\in \mathbb{R}^{q\times r}\times\mathbb{S}^{q}_{++}\times \mathbb{S}^{r}_{++}$ satisfying $\el(C,Q,R)\supseteq\mathcal{T}$.

In this study, we fix the center $C$ of the approximating matrix ellipsoid to the geometric center of $\mathcal{T}$.
Since $\el(C_k,Q_k,R_k)=\{C_k\}\oplus\el^0(Q_k,R_k)$, the set $\mathcal{T}$ can be written as
\begin{equation}
    \mathcal{T}=\left\{\sum_{k=1}^K F_kC_kG_k\right\}\oplus \mathcal{T}^0,\quad\mathcal{T}^0\coloneqq\left(\bigoplus_{k=1}^K F_k\el_k^0G_k\right),
\end{equation}
where $\el^0_k\coloneqq\el^0(Q_k,R_k)$.
Thus, we can fix the center as $C=\sum_{k=1}^K F_kC_kG_k$. With this choice of $C$, the outer approximation condition $\el(C,Q,R)\supseteq\mathcal{T}$ becomes equivalent to the inclusion condition centered at the origin:
\begin{equation}\label{apx:prob_0}
    \el^0(Q,R)\supseteq \mathcal{T}^0.
\end{equation}

Our objective is to find the optimal matrix ellipsoid that minimizes a specific size criterion.
As discussed in Section~\ref{subsec:matellip}, the size of the matrix ellipsoid $\el(C,Q,R)$ can be expressed as a function of $R\otimes Q$. Let $F(R\otimes Q)$ denote the function determining the size of the matrix ellipsoid. We consider the problem of minimizing this function:
\begin{equation}\label{prb:prb}
    \min_{Q,R} F(R\otimes Q)\quad \text{subject to} \quad \eqref{apx:prob_0}.
\end{equation}
In this paper, we specifically focus on cases where $F(X)= \tr(X)$ or $F(X)= \logdet(X)$.
Similar to vector-valued ellipsoids, setting $F(X)= \tr(X)$ corresponds to minimizing the sum of squared semi-axes. Alternatively, setting $F(X)= \logdet(X)$ implies volume minimization.
%Our goal is to obtain the optimal solution to \eqref{prb:prb} for each criterion.

For the subsequent analysis, we assume the linear transformation matrices satisfy the following condition.

\begin{ass}\label{ass:FG}
For all $k$, $F_k$ and $G_k$ are non-zero. Further, $F\coloneqq [F_1~~F_2\cdots F_K]$ is of full row rank, and $G\coloneqq [G_1^\top~~G_2^\top\cdots G_K^\top]^\top$ is of full column rank.
\end{ass}

This assumption is not restrictive.
If either $F_k$ or $G_k$ is zero, we have $F_k\el_kG_k=\{0\}$, which can simply be excluded from~\eqref{apx:prob_0}.
Moreover, if $q'=\mathrm{rank}\,F<q$, there exists a matrix $\bar{F}\inX{R}{q\times q'}$ such that $F_k=\bar{F}\hat{F}_k$ for all $k$. In this case, letting $\hat{F}\coloneqq [\hat{F}_1~~\hat{F}_2\cdots\hat{F}_{K}]$ yields $F=\bar{F}\hat{F}$. Since $\rank F=\rank(\bar{F}\hat{F})\le\rank \hat{F}$, the matrix $\hat{F}$ is of full row rank. Therefore, instead of \eqref{apx:prob_0}, we can consider the problem of approximating $\bigoplus_{k=1}^K \hat{F}_k\el_k^0G_k$. If $G$ is not of full column rank, we can similarly reduce the problem to an approximation in a lower dimension.

%In the following, we parameterize the family of matrix ellipsoids satisfying \eqref{apx:prob_0} and then solve \eqref{prb:prb}.

\section{Parameterization of Bounding Matrix Ellipsoids}\label{sec:parametrize}

\subsection{Bounding Conditions via LMIs}

We derive sufficient conditions for the outer approximation \eqref{apx:prob_0}. In this subsection, extending the method in \cite{boyd1994linear}, we obtain LMI-based sufficient conditions for the approximation.

\begin{lem}\label{lem:Spro}
    Let $(Q,R)\in \Spr{q}{r}$.
    Equation~\eqref{apx:prob_0} holds 
    if there exists $\alpha\inX{R}{K}_+$ such that
    \begin{equation}\label{apx:Sprocedure}
        R-X_{\rm sum}^\top Q^{-1}X_{\rm sum}\succeq\sum_{k=1}^K\alpha_kG_k^\top(R_k-X_k^\top Q^{-1}_kX_k)G_k,
    \end{equation}
    holds for any $X_k\in\el_k^0$
    where $X_{\rm sum}\coloneqq \sum_{k=1}^K F_kX_kG_k$.
\end{lem}
\begin{proof}
    For \eqref{apx:prob_0} to hold, it is sufficient that $X_{\rm sum}\in\el^0(Q,R)$ for any $X_k\in\el^0_k$, $k\in[1,K]$. When $X_k\in\el^0_k$ holds for all $k$, the right-hand side of \eqref{apx:Sprocedure} is positive semidefinite. Consequently, we have $R-X_{\rm sum}^\top Q^{-1}X_{\rm sum}\succeq 0$, which implies $X_{\rm sum}\in\el^0(Q,R)$.
\end{proof}

By introducing an augmented matrix $Z=\diag_{k\in[1,K]}(X_k) G\inX{R}{\sum_{k=1}^Kq_k\times r}$ and rewriting \eqref{apx:Sprocedure}, we obtain the following LMI representation.

\begin{prp}\label{prp:LMI}
    Let $(P,R)\in\Spr{q}{r}$.
    If there exists $\alpha\inX{R}{K}_+$ such that
    \begin{subequations}\label{apx:LMI}
    \begin{align}
        &\mathrm{diag}_{k\in[1,K]}(\alpha_k Q_k^{-1})-F^\top P F\succeq 0, \label{apx:LMI_P}\noeqref{apx:LMI_P}\\
        &\textstyle{
        R-\sum_{k=1}^K \alpha_kG_k^\top R_kG_k\succeq 0, \label{apx:LMI_R}\noeqref{apx:LMI_R}
        }
    \end{align}
    \end{subequations}
    then $Q=P^{-1}$ and $R$ satisfy \eqref{apx:prob_0}.
\end{prp}
\begin{proof}
    Regarding the augmented matrix $Z$, we have $X_{\rm sum}=FZ$ and $X_kG_k=U_kZ$, where $U_k\coloneqq [0_{q_k,\sum_{l=1}^{k-1}q_l}~~I_{q_k}~~0_{q_k,\sum_{l=k+1}^K q_l}]$. Substituting these relations into \eqref{apx:Sprocedure} yields{\small
    \begin{equation}
            R-Z^\top F^\top Q^{-1}FZ- \sum_{k=1}^K\alpha_k(G_k^\top R_kG_k-Z^\top U_k^\top Q_k^{-1}U_kZ)\succeq 0.
    \end{equation}}Substituting $Q^{-1}=P$ and rearranging the terms, this inequality is further rewritten as{\small
    \begin{equation}\label{apx:prf_Z}
        \begin{split}
            &\textstyle{
            R-\sum_{k=1}^K \alpha_kG_k^\top R_kG_k
            }\\
            &+Z^\top\left(\mathrm{diag}_{k\in[1,K]}(\alpha_k Q_k^{-1})-F^\top P F\right)Z\succeq 0.
        \end{split}
    \end{equation}}When \eqref{apx:LMI_P} and \eqref{apx:LMI_R} hold, \eqref{apx:prf_Z} is clearly satisfied. Therefore, \eqref{apx:Sprocedure} holds, which implies \eqref{apx:prob_0} by Lemma~\ref{lem:Spro}.
\end{proof}

Proposition~\ref{prp:LMI} characterizes the bounding matrix ellipsoid $\el^0(Q,R)$ that achieves the outer approximation as the solution to LMIs. By defining the solution set of these LMIs as
\begin{equation}\label{apx:dfn_F}
    \mathcal{F}\coloneqq \left\{(P,R)\in\Spr{q}{r}\;\middle|\; \exists \alpha\inX{R}{K}_+ \st \eqref{apx:LMI}\right\},
\end{equation}
the family of approximating matrix ellipsoids satisfying \eqref{apx:prob_0} can be written as $\{\el^0(P^{-1},R)\mid(P,R)\in \mathcal{F}\}$.

By applying the Schur complement to \eqref{apx:LMI_P} and rewriting it as a condition on $Q= P^{-1}$, we can obtain an alternative representation of $\mathcal{F}$.

\begin{cor}\label{cor:QR}
Let Assumption~\ref{ass:FG} hold.
If $(P,R)\in\mathcal{F}$, $Q=P^{-1}$ and $R$ satisfy
\begin{align}
\textstyle{
    Q-\sum_{k=1}^K\frac{1}{\alpha_k}F_kQ_kF_k^\top\succeq 0,\label{apx:MI_Q}
    }
\end{align}
and \eqref{apx:LMI_R} with some $\alpha\inX{R}{K}_{++}$.
Conversely, if $Q\inX{S}{q}$ and $R\inX{S}{r}$ satisfy \eqref{apx:MI_Q} and \eqref{apx:LMI_R} with some $\alpha\inX{R}{K}_{++}$, then $Q\succ 0$ and $(Q^{-1},R)\in\mathcal{F}$.
\end{cor}
\begin{proof}
    For the first statement, suppose $(P,R)\in\mathcal{F}$. Then, there exists $\alpha\inX{R}{K}_+$ satisfying \eqref{apx:LMI_P} and \eqref{apx:LMI_R}. 
    We show that this $\alpha$ belongs to $\mathbb{R}^K_{++}$ and satisfies \eqref{apx:MI_Q}.
    Regarding the block diagonal elements $\alpha_kQ_k^{-1}-F_k^\top PF_k\succeq 0$ in \eqref{apx:LMI_P}, Assumption~\ref{ass:FG} ensures that $F_k$ is non-zero. Since $P\succ 0$, the term $-F_k^\top P F_k$ has a negative eigenvalue. Thus, we must have $\alpha_k>0$ for all $k\in[1,K]$.
    Furthermore, since $P\succ 0$, \eqref{apx:LMI_P} is equivalent to
    \begin{equation}\label{PrfCorQR2}
        \sqmat{\diag_{k\in[1,K]}(\alpha_kQ_k^{-1})}{F^\top}{F}{Q}\succeq 0
    \end{equation}
    by the Schur complement. Applying the Schur complement to the top-left block yields \eqref{apx:MI_Q}. Note that $\diag_{k\in[1,K]}(\alpha_kQ_k^{-1})$ is invertible because $\alpha_k>0$ and $Q_k\succ 0$.
    
    Next, we prove the second statement. Suppose $(Q,R)$ satisfies \eqref{apx:MI_Q} and \eqref{apx:LMI_R} with some $\alpha\inX{R}{K}_{++}$. We show that $Q\succ0$, $R\succ 0$, and $P=Q^{-1}$ satisfies \eqref{apx:LMI_P}.
    First, we show $Q\succ 0$. Inequality \eqref{apx:MI_Q} can be rewritten as
    \begin{equation}\label{PrfCorQR1}
    \textstyle{
        Q-F\diag_{k\in[1,K]}\left(\frac{1}{\alpha_k}Q_k\right)F^\top\succeq 0.
        }
    \end{equation}
    Since $\alpha_k>0$, $Q_k\succ 0$, and $F$ is of full row rank by Assumption~\ref{ass:FG}, we obtain $Q\succ 0$. Similarly, $R\succ 0$ follows from \eqref{apx:LMI_R} and the full column rank of $G$.
    Next, applying the Schur complement to \eqref{PrfCorQR1} yields \eqref{PrfCorQR2}. Applying the Schur complement again to the bottom-right block gives \eqref{apx:LMI_P}.
\end{proof}

\subsection{Parameterization of Bounding Ellipsoids}
Based on Corollary~\ref{cor:QR}, we parameterize the bounding matrix ellipsoid using $\alpha\inX{R}{K}_{++}$ as
\begin{align}\label{apx:parametrize}
    \hspace{-10pt}Q(\alpha)=\sum_{k=1}^K\frac{1}{\alpha_k}F_kQ_kF_k^\top,\ 
    R(\alpha)=\sum_{k=1}^K \alpha_kG_k^\top R_kG_k.
\end{align}
This parameterization constitutes the boundary of the bounding matrix ellipsoids represented by the LMIs.

\begin{thm}\label{thm:parametrize}
    Let Assumption~\ref{ass:FG} hold.
    Define $Q(\alpha)$ and $R(\alpha)$ as in \eqref{apx:parametrize}.
    Then, $\el^0(Q(\alpha),R(\alpha))\supseteq\mathcal{T}^0$ for any $\alpha\inX{R}{K}_{++}$.
    Moreover, if $(Q^{-1},R)\in\mathcal{F}$, then there exists $\alpha\inX{R}{K}_{++}$ such that $\el^0(Q(\alpha),R(\alpha))\subseteq\el^0(Q,R)$.
\end{thm}
\begin{proof}
    The first statement is evident from Corollary~\ref{cor:QR}. We prove the second statement.
    Suppose $(Q^{-1},R)\in\mathcal{F}$. By Corollary~\ref{cor:QR}, there exists $\alpha\inX{R}{K}_{++}$ satisfying \eqref{apx:MI_Q} and \eqref{apx:LMI_R}. By definition, these inequalities are equivalent to $Q(\alpha)\preceq Q$ and $R(\alpha)\preceq R$, respectively. From the definition of matrix ellipsoids, we obtain $\el^0(Q(\alpha),R(\alpha))\subseteq\el^0(Q,R)$.
\end{proof}

The parameterization in \eqref{apx:parametrize} serves as a natural extension of the vector-valued ellipsoid parameterization proposed in \cite{ellip:durieu_multi-input_2001}.
Indeed, by setting $r=1$ and fixing the matrices as $R(\alpha)=G_k^\top R_k G_k=1$, \eqref{apx:parametrize} reduces to
\begin{align}\label{apx:parametrize_vec}
\textstyle{
    Q(\alpha)=\sum_{k=1}^K\frac{1}{\alpha_k}F_kQ_kF_k^\top,\quad
    \sum_{k=1}^K \alpha_k =1,
    }
\end{align}
which is equivalent to~\cite[Eq. (36)]{ellip:durieu_multi-input_2001}.

\section{Optimal Matrix Ellipsoid}\label{sec:opt}
We consider finding the optimal matrix ellipsoid, with respect to the trace or log-determinant criterion, among the matrix ellipsoids parameterized by \eqref{apx:parametrize}.
Instead of \eqref{prb:prb}, we consider the following problem:
\begin{equation}\label{opt:opt}
   \min_{\alpha\in\mathbb{R}^K_{++}} f(\alpha)=F(R(\alpha)\otimes Q(\alpha)).
\end{equation}
An important property is its invariance to the scaling of $\alpha$, i.e., $f(c\alpha)=f(\alpha)$ holds for any $c>0$.
This originates from the fact that, by \eqref{apx:parametrize} and the properties of the Kronecker product, we have
\begin{equation}
   R(c\alpha)\otimes Q(c\alpha)=(cR(\alpha))\otimes(c^{-1}Q(\alpha))=R(\alpha)\otimes Q(\alpha).
\end{equation}
Due to this invariance, we can fix the scale of the parameter $\alpha$ without loss of generality.
Introducing the linear constraint $\tr(R(\alpha))=r$, we hereafter consider the normalized problem:
\begin{equation}\label{opt:opt_normarized}
   \min_{\alpha\in\mathcal{S}} f(\alpha)=F(R(\alpha)\otimes Q(\alpha)),
\end{equation}
where
\begin{equation}
   \mathcal{S}=\left\{\alpha\in\mathbb{R}^K_{++}\;\middle|\; \tr(R(\alpha))=\sum_{k=1}^K\alpha_k \tr(G_k^\top R_k G_k)=r\right\}.
\end{equation}
%Hereafter, we consider the normalized problem \eqref{opt:opt_normarized} instead of \eqref{opt:opt}.

\subsection{Trace Criterion}
First, we consider the case where $F(X)=\tr(X)$ in \eqref{opt:opt_normarized}.
We transform the objective function as
\begin{align}
   f(\alpha)=\tr(R(\alpha)\otimes Q(\alpha))=\tr(R(\alpha))\tr (Q(\alpha)).
\end{align}
Under the condition $\alpha\in\mathcal{S}$, the term $\tr(R(\alpha))$ is a constant.
Consequently, \eqref{opt:opt_normarized} reduces to the following convex program:
\begin{equation}\label{opt:trace}
\textstyle{
   \min_{\alpha\in\mathbb{R}^K_{++}} \sum_{k=1}^K \frac{u_k}{\alpha_k} \quad
   \text{subject to } \sum_{k=1}^K v_k\alpha_k=1,
   }
\end{equation}
where $u_k=\tr(F_kQ_kF_k^\top)$ and $v_k=r^{-1}\tr(G_k^\top R_kG_k)$.
%Since~\eqref{opt:trace} shares a common structure with the vector-valued ellipsoid case~\cite[Theorem 4.4]{ellip:durieu_multi-input_2001},
%the analytic solution can be obtained using the method of Lagrange multipliers.
We can derive its analytic solution.

\begin{prp}\label{prp:trace}
    Suppose $F(X)=\tr(X)$ in~\eqref{opt:opt_normarized}.
    Then, the optimal solution $\alpha^\star\in\mathcal{S}$ is given by
    \begin{equation}
        \alpha^\star_k=\frac{r\sqrt{\tr(F_kQ_kF_k^\top)/\tr(G_k^\top R_k G_k)}}{\sum_{l=1}^K\sqrt{\tr(F_lQ_lF_l^\top) \tr(G_l^\top R_l G_l)}}.
    \end{equation}
\end{prp}
\begin{proof}
    This result is derived via the method of Lagrange multipliers.
    The details are omitted as they are similar to those for the vector-valued case~\cite[Theorem 4.4]{ellip:durieu_multi-input_2001}.
\end{proof}

Using this analytic solution, the optimal bounding ellipsoid $\el^0(Q(\alpha^\star), R(\alpha^\star))$ can be instantly identified from the parameterized family without relying on complex LMIs or nonlinear optimization solvers.
Since it requires only matrix multiplications and trace computations without the need for iterative optimization, this analytical solution is an efficient tool for accurately obtaining the bounding matrix ellipsoid with minimal computational cost.

\subsection{Log-Determinant Criterion}\label{subsec:Det}

Consider the case where $F(X)=\logdet(X)$ in \eqref{opt:opt_normarized}.
For vector-valued ellipsoids, \eqref{opt:opt_normarized} easily reduces to a convex program~\cite[Theorem 4.3]{ellip:durieu_multi-input_2001}.
However, in contrast, the objective function is non-convex in the matrix-valued case, making it difficult to obtain the global optimum.
To address this issue, we develop an algorithm to find a stationary point based on the MM framework, leveraging the concavity of the log-determinant function.

The objective function $f(\alpha)$ can be written as follows:
\begin{equation}\label{opt:fdecom}
    \begin{split}
        f(\alpha)&=\logdet(R(\alpha)\otimes Q(\alpha))\\
        &=
        q\, \underbrace{\logdet(R(\alpha))}_{=: f_R(\alpha)}+r\, \underbrace{\logdet(Q(\alpha))}_{=:f_Q(\alpha)}.
    \end{split}
\end{equation}
As the eigenvalues of $Q(\alpha)$ or $R(\alpha)$ approach $0$, $f(\alpha)$ diverges to negative infinity.
This indicates that the volume of $\el(C,Q(\alpha),R(\alpha))$ degenerates to $0$.
To avoid this situation and ensure that $f(\alpha)$ takes a bounded value for any $\alpha\inX{R}{K}_{++}$, we introduce the following assumption.

\begin{ass}\label{ass:H}
    There exists a positive constant $\lambda_{\rm min}>0$ such that $R(\alpha)\otimes Q(\alpha)\succ\lambda_{\rm min}I_{qr}$ holds for any $\alpha\in\mathbb{R}^K_{++}$.
\end{ass}

Under this assumption, it clearly holds that $f(\alpha)> qr \log(\lambda_{\rm min})$.

\subsubsection{Derivation of MM Algorithm}

Under Assumption~\ref{ass:H}, we consider the optimization problem~\eqref{opt:opt_normarized}.
As shown in \cite[Theorem 4.3]{ellip:durieu_multi-input_2001}, the function $f_Q(\alpha)$ is convex.
However, since $f_R(\alpha)$ is the composition of the concave function $\logdet(X)$ and the linear mapping $R(\alpha)$ with respect to $\alpha$, it is concave overall~\cite[Sec. 3.2.3]{boyd2004convex}.
%It is difficult to completely avoid the non-convexity introduced by $f_R(\alpha)$ within the standard convex optimization framework.
Note that, in the vector-valued case ($r=1$), convexity is guaranteed because imposing the constraint $\alpha \in \mathcal{S}$ fixes $R(\alpha) = 1$ and yields $f_R(\alpha) = 0$. In contrast, for matrix representations with $r \ge 2$, the concavity of $f_R(\alpha)$ remains even under the constraint $\alpha \in \mathcal{S}$.

%Therefore, we propose an algorithm that sequentially updates the solution by constructing a surrogate function that upper-bounds $f(\alpha)$ via the first-order approximation of the log-determinant function, and minimizing it.
%Since the log-determinant function is concave, its first-order approximation at any point provides an upper bound on the original function (Proposition~\ref{aprp:logdet} in the Appendix).
%Applying this to $f_R(\alpha)$ and $f_Q(\alpha)$ yields a surrogate function that upper-bounds $f(\alpha)$.

Due to this inherent non-convexity, finding a global optimal solution is generally intractable. To address this, we develop an iterative algorithm based on the MM framework that guarantees convergence to the set of stationary points. Specifically, we propose an algorithm that sequentially updates the solution by constructing and minimizing a surrogate function that upper-bounds the original objective $f(\alpha)$. The construction of this surrogate exploits the concavity of the log-determinant function; its first-order approximation at any given point naturally provides a tight upper bound on the original function (see Proposition~\ref{aprp:logdet} in the Appendix). By applying this first-order approximation to $f_R(\alpha)$ and combining it with $f_Q(\alpha)$, we obtain a tractable surrogate function that bounds $f(\alpha)$ from above at each iteration.

\begin{prp}\label{prp:surrogate_function}
    Define a function $g(\cdot|\alphat{t}): \mathbb{R}^K_{++}\rightarrow \mathbb{R}$ as
    \begin{equation}
    \textstyle{
        g(\alpha|\alphat{t})\coloneqq \sum_{k=1}^K\left(\Akt\alpha_k +\frac{\Bkt}{\alpha_k}\right)+\Ct,
        }
    \end{equation}
    where we denote $\Rt=R(\alphat{t})$, $\Qt=Q(\alphat{t})$, $\Akt=q\tr\left({\Rt}^{-1}G_k^\top R_kG_k\right)$, $\Bkt=r\tr\left({\Qt}^{-1}F_kQ_kF_k^\top\right)$ and $\Ct = f(\alphat{t})-2qr$.
    % \begin{align}
    %     \Akt&=q\tr\left({\Rt}^{-1}G_k^\top R_kG_k\right),\\ 
    %     \Bkt&=r\tr\left({\Qt}^{-1}F_kQ_kF_k^\top\right),
    % \end{align}
    % and $\Ct = f(\alphat{t})-2qr$.
    Then, $g(\alpha|\alphat{t})$ satisfies
    \begin{equation}\label{opt:surrogate}
        g(\alpha|\alphat{t})\geq f(\alpha).
    \end{equation}
    Furthermore, the equality in \eqref{opt:surrogate} holds when $\alpha=\alphat{t}$.
    % Moreover, when $\alpha=\alphat{t}$, it holds that
    % \begin{equation}
    %     g(\alphat{t}|\alphat{t})=f(\alphat{t}),\quad \nabla g(x|\alphat{t})\rvert_{x=\alphat{t}}=\nabla f(\alphat{t}).
    % \end{equation}
\end{prp}
\begin{proof}
    Applying Proposition~\ref{aprp:logdet} to $f_R(\alpha)$, we obtain
    \begin{align}
        f_R(\alpha)\le f_R(\alphat{t})+\tr\left({\Rt}^{-1}(R(\alpha)-\Rt)\right).\label{prf:frfirst}
    \end{align}
    Since $\tr({\Rt}^{-1}R(\alpha))=\sum_{k=1}^K \Akt\alpha_k/q$, it holds that
    \begin{align}
    \textstyle{
        f_R(\alpha)\le \sum_{k=1}^K\left(\frac{\Akt\alpha_k}{q}\right) +f_R(\alphat{t})-r.\label{prf:fr}
        }
    \end{align}
    Similarly, for $f_Q(\alpha)$, we have
    \begin{align}\label{prf:fq}
    \textstyle{
        f_Q(\alpha)\le\sum_{k=1}^K \left(\frac{\Bkt}{r\alpha_k}\right)+f_Q(\alphat{t})-q.
        }
    \end{align}
    Since $f(\alpha)=qf_R(\alpha)+rf_Q(\alpha)$, \eqref{opt:surrogate} follows.
    Moreover, when $\alpha=\alphat{t}$, the equality in \eqref{prf:frfirst} clearly holds.
    Since the equality also holds in \eqref{prf:fq}, it holds in \eqref{opt:surrogate} as well.
\end{proof}

In the MM algorithm~\cite{math:sum2017MMalgo}, minimizing the surrogate function $g(\alpha|\alphat{t})$ monotonically decreases the original objective function $f(\alpha)$.
The function $g(\alpha|\alphat{t})$ possesses a highly favorable structure for optimization.
Specifically, it separates into a sum of mutually independent strongly convex functions $\Akt\alpha_k+\Bkt/\alpha_k$ with respect to each element $\alpha_k$.
However, minimizing this under the constraint of the original feasible region $\mathcal{S}$ destroys the decoupling among the variables $\alpha_k$.
Therefore, we propose an approach that analytically solves the unconstrained minimization problem for $g(\alpha|\alphat{t})$ and then projects the solution onto $\mathcal{S}$.
Notably, this projection can be performed without altering the value of the function $f$.

First, for the unconstrained optimization of $g$, the minimizer can be obtained analytically.

\begin{prp}\label{prp:unconstrained_opt}
    Consider the optimization problem $\min_{\alpha\in\mathbb{R}^K_{++}}g(\alpha|\alphat{t})$.
    The minimizer $\talphat$ of this problem is obtained by
    \begin{equation}\label{opt:talpha}
        \talphat_k =\sqrt{\frac{\Bkt}{\Akt}}=\sqrt{\frac{r\tr\left({\Qt}^{-1}F_kQ_kF_k^\top\right)}{q\tr\left({\Rt}^{-1}G_k^\top R_kG_k\right)}}.
    \end{equation}
\end{prp}
\begin{proof}
    Since $g$ is strongly convex, the analytical solution \eqref{opt:talpha} 
    is obtained by solving the first-order condition:
    \begin{equation}\label{prf:diff_g}
        \frac{\partial g(\alpha|\alphat{t})}{\partial\alpha_k}=\Akt-\frac{\Bkt}{\alpha_k^2}=0.
    \end{equation}
\end{proof}

Next, exploiting the scale-invariance of the objective function $f(c\alpha)=f(\alpha)$, we choose $\alphat{t+1}\in\mathcal{S}$ such that $f(\alphat{t+1})=f(\talphat)$:
\begin{equation}\label{opt:proj}
    \alphat{t+1}=\frac{r}{\tr(R(\talphat))}\talphat.
\end{equation}
Consequently, we propose the optimization procedure detailed in Algorithm~\ref{alg:opt_convergence}.

\begin{algorithm}[t]
    \caption{MM-based Optimization}
    \label{alg:opt_convergence}
    \SetKwInOut{Input}{Input}
    \SetKwInOut{Output}{Output}

    \Input{Initial point $\alpha^{(0)} \in \mathcal{S}$, tolerance $\epsilon > 0$}
    \Output{Optimized parameter $\alpha^{(t)}$}

    $t \gets 0$\;
    Compute initial objective value $f(\alpha^{(0)})$\;
    
    \Repeat{$|f(\alpha^{(t)}) - f(\alpha^{(t-1)})| < \epsilon$}{
        $t \gets t + 1$\;
        Update $\alpha^{(t)}$ according to \eqref{opt:talpha} and \eqref{opt:proj} \label{line:eng_update}\;
        Compute current objective value $f(\alpha^{(t)})$\;
    }
    \Return $\alpha^{(t)}$\;
\end{algorithm}

The primary advantage of Algorithm~\ref{alg:opt_convergence} is that the update rules \eqref{opt:talpha} and \eqref{opt:proj} are given in closed form, allowing instant computation from the current parameter $\alpha^{(t)}$.
Compared to first-order methods like projected gradient descent, \eqref{opt:talpha} eliminates the need for line searches to determine the step size.
Furthermore, compared to interior-point or Newton methods, which typically require $O(K^3)$ operations per iteration due to Hessian construction and linear equation solving, our proposed update relies solely on matrix multiplications and trace computations, reducing the computational complexity per iteration to $O(K)$.
This dramatically improves scalability with respect to the number of matrix ellipsoids $K$.
Note that, in Algorithm~\ref{alg:opt_convergence}, the choice of the initial point strongly affects its behavior.
A reasonable choice for the initial point is the analytical solution obtained under the trace criterion in Proposition~\ref{prp:trace}.

\subsubsection{Convergence of Algorithm~\ref{alg:opt_convergence}}
We discuss the convergence properties of Algorithm~\ref{alg:opt_convergence}.
In this subsection, we prove that the sequence of objective values $\{f(\alphat{t})\}$ generated by the algorithm monotonically decreases and converges to a finite limit, and that the parameter sequence $\{\alphat{t}\}$ asymptotically approaches the set of stationary points defined by $\Gamma\coloneqq \{\alpha\in\mathcal{S}\mid \nabla f(\alpha)=0\}.$
%\begin{equation}
%    \Gamma\coloneqq \{\alpha\in\mathcal{S}\mid \nabla f(\alpha)=0\}.
%\end{equation}
We present the following theorem.

\begin{thm}\label{thm:convergence}
    Let Assumption~\ref{ass:FG} and Assumption~\ref{ass:H} hold.
    The sequence $\{\alphat{t}\}_{t=0}^\infty$ generated by \eqref{opt:talpha} and \eqref{opt:proj} from an initial point $\alphat{0}$ asymptotically approaches the stationary point set $\Gamma$:
    \begin{equation}\label{opt:convergence}
        \lim_{t \to \infty} \inf_{\alpha^* \in \Gamma} \|\alpha^{(t)} - \alpha^*\| = 0.
    \end{equation}
    Furthermore, the sequence of objective values $\{f(\alpha^{(t)})\}_{t=0}^\infty$ monotonically decreases and converges to a finite limit. % $f^* \ge f_{\min}$.
\end{thm}

We prove this theorem based on Zangwill's Global Convergence Theorem~\cite[Sec. 7.6]{math:luenberger2021linearnonlinear}.
As a preparation, we first establish the compactness of the sublevel set.

\begin{lem}\label{lem:compact}
    Let Assumption~\ref{ass:FG} and Assumption~\ref{ass:H} hold.
    Let $\alphat{0}\in\mathcal{S}$.
    Define a sublevel set as $\Omega\coloneqq \{ \alpha\in\mathcal{S}\mid f(\alpha)\le f(\alphat{0})\}$.
    Then, $\Omega$ is a bounded and closed set. % (i.e., compact).
\end{lem}
\begin{proof}
    Since $\mathcal{S}$ is bounded, $\Omega$ is clearly bounded.
    To prove closedness, let $\{\betat\}_{t=1}^\infty \subseteq \Omega$ be a sequence converging to $\bar{\beta}$.
    By continuity, $f(\bar{\beta})\le f(\alphat{0})$ and $\tr(R(\bar{\beta}))=r$ hold,
    so it suffices to show $\bar{\beta}\in\mathbb{R}^K_{++}$.
    % The boundedness of $\Omega$ is obvious since $\mathcal{S}$ is bounded.
    % We now show that $\Omega$ is closed.
    % Suppose a sequence $\{\betat\}_{t=1}^\infty \subseteq \Omega$ converges to $\bar{\beta}\coloneqq\lim_{t\rightarrow\infty}\betat$.
    % We show $\bar{\beta}\in\Omega$.
    % Since $f(\alpha)$ and $\tr(R(\alpha))$ are continuous, $f(\bar{\beta})\le f(\alphat{0})$ and $\tr(R(\bar{\beta}))=r$ hold.
    % Therefore, it suffices to show $\bar{\beta}\in\mathbb{R}^K_{++}$.
    Assume $\bar{\beta}_i = 0$ for some $i$ to derive a contradiction.
    Since the constraint $\tr(R(\bar{\beta})) = r$ holds, there must exist some $j\neq i$ such that $\bar{\beta}_j > 0$.
    Expanding the matrix $P(\betat)\coloneqq R(\betat)\otimes Q(\betat)$, we have
    \begin{align}
        P(\betat) &=\sum_{k=1}^K\sum_{l=1}^K \frac{\betat_l}{\betat_k} (G_l^\top R_lG_l) \otimes (F_kQ_kF_k^\top)\\
        &\succeq \frac{\betat_j}{\betat_i} (G_j^\top R_jG_j) \otimes (F_iQ_iF_i^\top).
    \end{align}
    As $t \rightarrow \infty$, the ratio $\betat_j/\betat_i \rightarrow \infty$.
    By Assumption~\ref{ass:FG}, the right-hand side is non-zero; thus, its maximum eigenvalue diverges to $\infty$ as $t\rightarrow\infty$, which consequently drives the maximum eigenvalue of $P(\betat)$ to $\infty$.
    On the other hand, Assumption~\ref{ass:H} guarantees that all eigenvalues of $P(\betat)$ are bounded from below by $\lambda_{\rm min} > 0$ for any $t$.
    Since the determinant is the product of all eigenvalues, we obtain $\lim_{t \rightarrow \infty} \det(P(\betat)) = \infty$.
    As a result, the objective function diverges, i.e.,
    $\lim_{t \rightarrow \infty} f(\betat) = \lim_{t \rightarrow \infty} \logdet(P(\betat)) = \infty.$
%    \begin{equation}
%        \lim_{t \rightarrow \infty} f(\betat) = \lim_{t \rightarrow \infty} \logdet(P(\betat)) = \infty.
%    \end{equation}
    However, since $\betat \in\Omega$, we must have $f(\betat) \le f(\alphat{0})$ for all $t$, which is a contradiction.
    Thus, the initial hypothesis is false, proving that $\bar{\beta} \in \mathbb{R}^K_{++}$.
    Since $\bar{\beta}\in\Omega$, $\Omega$ is a closed set.
\end{proof}

Furthermore, we show the monotonic decrease of the objective function sequence.

\begin{lem}\label{lem:monotonedecrease}
    Let $\alphat{t}\in\mathcal{S}$ and suppose $\alphat{t+1}\in\mathcal{S}$ is generated according to \eqref{opt:talpha} and \eqref{opt:proj}.
    Then, $f(\alphat{t+1})\le f(\alphat{t})$.
    Moreover, if $\alphat{t}\notin \Gamma$, then $f(\alphat{t+1})< f(\alphat{t})$.
\end{lem}
\begin{proof}
    First, we show that the gradients of $f(\alpha)$ and $g(\alpha|\alphat{t})$ match at $\alpha=\alphat{t}$.
    From the derivative of the log-determinant function~\cite[Lemma 3.2]{ellip:durieu_multi-input_2001}, we have{\small
    \begin{equation}
            \frac{\partial f(\alpha)}{\partial \alpha_k}=q\tr(R(\alpha)^{-1}G_k^\top R_kG_k)
            -\frac{r}{\alpha_k^2}\tr(Q(\alpha)^{-1}F_k Q_kF_k^\top).
    \end{equation}}Also, the derivative of $g(\alpha|\alphat{t})$ is given by \eqref{prf:diff_g}.
    From these equations, we obtain
    % \begin{equation}\label{prf:grad_fg}
        $\nabla f(\alphat{t})=\nabla_x g(x|\alphat{t})\rvert_{x=\alphat{t}}$.
    % \end{equation}

    Now, we show that $f(\alphat{t+1}) \le f(\alphat{t})$.
    Let $\talphat$ be defined by \eqref{opt:talpha}.
    By Proposition~\ref{prp:unconstrained_opt}, $\talphat$ minimizes $g(\alpha|\alphat{t})$, yielding
    $g(\talphat|\alphat{t})\le g(\alphat{t}|\alphat{t})$.
    Furthermore, the surrogate function property in Proposition~\ref{prp:surrogate_function} implies
    \begin{equation}\label{prf:decrease}
        f(\talphat) \le g(\talphat|\alphat{t})\le g(\alphat{t}|\alphat{t}) = f(\alphat{t}).
    \end{equation}
    Because $f(\alpha)$ is invariant to positive scaling, \eqref{opt:proj} ensures $f(\alphat{t+1})=f(\talphat)$.
    Combining this with \eqref{prf:decrease} yields $f(\alphat{t+1})\le f(\alphat{t})$.

    % In particular, when $\alphat{t}\notin \Gamma$, we have $\nabla f(\alphat{t})\neq 0$.
    % Relation \eqref{prf:grad_fg} implies $\nabla_x g(x|\alphat{t})\rvert_{x=\alphat{t}}\neq0$.
    % Thus, $\alphat{t}$ does not minimize $g(\alpha|\alphat{t})$, meaning $\talphat\neq\alphat{t}$ and $g(\talphat|\alphat{t})<g(\alphat{t}|\alphat{t})$.
    % Consequently, $f(\alphat{t+1})< f(\alphat{t})$ follows from \eqref{prf:decrease}.

    In particular, when $\alphat{t}\notin \Gamma$, we have $\nabla f(\alphat{t})\neq 0$.
    Due to the gradient matching property, $\alphat{t}$ does not minimize the function $g(\alpha|\alphat{t})$, meaning $\talphat\neq\alphat{t}$.
    This makes the second inequality in \eqref{prf:decrease} strict, i.e., $g(\talphat|\alphat{t})<g(\alphat{t}|\alphat{t})$, yielding $f(\alphat{t+1})< f(\alphat{t})$.
\end{proof}

Building on these lemmas, we prove Theorem~\ref{thm:convergence}.

\begin{proof}
    By Lemma~\ref{lem:compact} and Lemma~\ref{lem:monotonedecrease}, all $\alphat{t}$ for $t \ge 0$ are contained in the compact set $\Omega$.
    Moreover, the update rule defined by \eqref{opt:talpha} and \eqref{opt:proj} is continuous.
    Given the strict monotonic decrease established in Lemma~\ref{lem:monotonedecrease}, Zangwill's Global Convergence Theorem~\cite[Sec. 7.6]{math:luenberger2021linearnonlinear} is applicable.
    Consequently, the limit of any convergent subsequence of $\{\alphat{t}\}_{t=0}^\infty$ belongs to $\Gamma$.
    If we assume $
        \lim_{t \to \infty} \inf_{\alpha^* \in \Gamma} \|\alpha^{(t)} - \alpha^*\| >\epsilon$ with a positive scalar $\epsilon>0$,
    then there would exist a subsequence that does not approach $\Gamma$, contradicting this fact.
    Thus, \eqref{opt:convergence} holds.

    Furthermore, the sequence $\{f(\alpha^{(t)})\}_{t=0}^\infty$ is bounded from below by Assumption~\ref{ass:H} and monotonically decreases by Lemma~\ref{lem:monotonedecrease}. Therefore, it converges to a finite limit.
\end{proof}

Note that when targeting vector-valued ellipsoids ($r=1$), the optimization problem \eqref{opt:opt_normarized} inherently reduces to minimizing a strongly convex function.
In this case, since a unique global optimum exists, $\Gamma$ becomes a singleton, and the sequence $\alphat{t}$ generated by the proposed algorithm rigorously converges to this global optimum.

\section{Numerical Comparison}\label{sec:numerical}
We compare the proposed Algorithm~\ref{alg:opt_convergence} with two baseline methods for minimizing the volume of the approximating matrix ellipsoid.
For vector-valued ellipsoids, it is known that such volume minimization can be solved as a convex optimization problem.
Therefore, under this specific problem setting $r=1$, we can conduct a fair and rigorous comparison of the proposed method with a nonlinear convex optimization solver and a semidefinite programming (SDP) solver.

\subsection{Baseline Methods}
%\subsection{Vector-Valued Ellipsoid Approximation}
%In this subsection, we consider vector-valued ellipsoids by setting $r=1$.
Without loss of generality, we can fix $G_k^\top R_k G_k=1$.
In this case, the minimization problem \eqref{opt:opt_normarized} reduces to the following nonlinear convex optimization problem:
\begin{align}\label{cmp:nlp}
\textstyle{
   \min_{\alpha\in\mathbb{R}^K_{++}}\logdet(Q(\alpha))\quad \text{s.t.\ } \sum_{k=1}^K \alpha_k=1.
   }
\end{align}
The convexity of the objective function is shown in \cite[Theorem 4.3]{ellip:durieu_multi-input_2001}.
This theorem also provides the analytical gradient and Hessian of the objective function, enabling the problem to be solved using a nonlinear interior-point method.
Alternatively, using the LMI representation in Proposition~\ref{prp:LMI}, the problem can also be formulated as an SDP:
\begin{equation}\label{cmp:sdp}
   \begin{split}
   \textstyle{
       \min_{P\inX{S}{q}_{++},\alpha\inX{R}{K}_+} -\logdet (P)\quad
   \text{s.t.\ } \eqref{apx:LMI_P}, \quad \sum_{k=1}^K \alpha_k=1.
   }
   \end{split}
\end{equation}
The difference in the sign of the objective function arises because of the relation $P=Q^{-1}$ in Proposition~\ref{prp:LMI}.
Although \eqref{cmp:sdp} involves more variables than \eqref{cmp:nlp}, it can be solved using methods that guarantee polynomial-time convergence.

In the next subsection, we compare Algorithm~\ref{alg:opt_convergence} with the following two methods:\\
\textbf{fmincon}: Solves \eqref{cmp:nlp} via MATLAB's nonlinear interior-point method using analytical gradients and Hessians.\\
\textbf{CVX (SDPT3)}: Solves \eqref{cmp:sdp} using the SDPT3 solver~\cite{Toh1999_SDPT3} via the MATLAB optimization interface CVX~\cite{cvx, gb08}\footnote{In SDPT3, using the $\sqrt[n]{\det(\cdot)}$ function is more efficient than using the $\logdet(\cdot)$ function; thus, we modify the objective function accordingly in the numerical experiments (https://cvxr.com/cvx/doc/advanced.html).}.

\subsection{Numerical Examples}
Experiments ran in MATLAB R2024a on an Intel Core i9 (2.8 GHz, 32 GB RAM). We set $F_k = I_q$ and generated $Q_k = N_k N_k^\top + 0.1 I_q$,
%We present the settings and results of the numerical experiments.
% All numerical experiments were conducted in MATLAB R2024a on a standard desktop PC equipped with an Intel Core i9 2.8 GHz processor and 32 GB of RAM.
% The linear transformation matrices were set to $F_k = I_q$. Each matrix $Q_k \in \mathbb{S}^q_{++}$ was randomly generated as $Q_k = N_k N_k^\top + 0.1 I_q$, 
where the elements of $N_k \in \mathbb{R}^{q \times q}$ were drawn from a standard normal distribution.
The initial point $\alphat{0}\in\mathcal{S}$ for Algorithm~\ref{alg:opt_convergence} and fmincon was randomly generated as $\alphat{0}=\beta/\sum_{k=1}^K\beta_k$, where $\beta\inX{R}{K}_{++}$ is a vector whose elements follow a uniform distribution between $0.1$ and $1$.
The computation times reported in the following experiments are the averages over $20$ independent trials.

In the first experiment, we compare the three methods in a small-scale setting with $q=5$ and $K \in \{10, 20, 30, 40, 50\}$.
The results confirm that all three methods converge to equivalent objective values.
A comparison of the computation times required for convergence is shown in Fig.~\ref{fig:smallexp}.
As is evident from the graph, the proposed method and fmincon achieve fast convergence, whereas the computation time for CVX (SDPT3) tend to increase as the number of matrix ellipsoids $K$ increased.
This is attributed to the constraint structure of the SDP \eqref{cmp:sdp}.
In \eqref{cmp:sdp}, not only is the number of variables proportional to $K$, but the dimension of \eqref{apx:LMI_P} is $qK$, which also grows proportionally with $K$.
Solving this with a standard interior-point method requires a computational complexity on the order of $O(K^4)$ or higher~\cite[Sec. 11.8.3]{boyd2004convex}, \cite[Chap. 6]{ben2001lectures}. Thus, computation time increases due to the inflated memory usage and matrix decomposition costs associated with handling massive LMI constraints.

Next, we conduct experiments in a larger-scale setting with $q=50$ and $K \in \{100, 200, 300, 400, 500\}$.
Since CVX (SDPT3) struggles to solve problems of this scale, the comparison is restricted to the proposed method and fmincon.
The computation times for each method are shown in Fig.~\ref{fig:bigexp}. The results show that the computation time of fmincon increases as $K$ grew.
In contrast, the proposed method shows almost no increase in computation time, reaching the optimal solution in an extremely short time even in the large-scale setting.
The excellent scalability of the proposed method with respect to $K$ is attributed to the fact that each solution update is executed in $O(K)$ operations.

% \begin{figure}[t]
%  \centerline{\includegraphics[width=\columnwidth]{small_exp.eps}}
%  \caption{Comparison with fmincon and CVX (SDPT3).}
%  \label{fig:smallexp}
% \end{figure}

% \begin{figure}[t]
%  \centerline{\includegraphics[width=\columnwidth]{big_exp.eps}}
%  \caption{Comparison with fmincon.}
%  \label{fig:bigexp}
% \end{figure}

\begin{figure}[t]
 \begin{minipage}[t]{0.49\columnwidth}
 \vspace{0pt}
  \centerline{\includegraphics[width=\columnwidth]{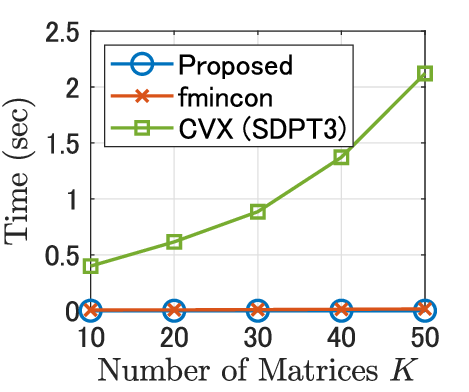}}
  \caption{Comparison with fmincon and CVX (SDPT3).}
  \label{fig:smallexp}
 \end{minipage}
 \hfill
 \begin{minipage}[t]{0.49\columnwidth}
 \vspace{0pt}
  \centerline{\includegraphics[width=\columnwidth]{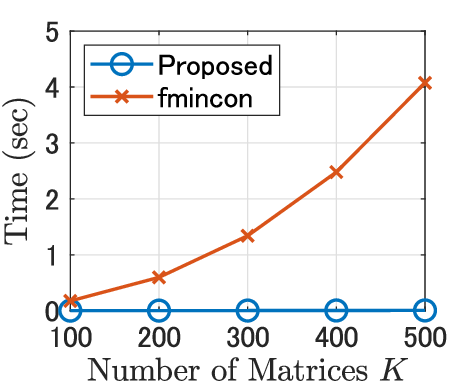}}
  \caption{Comparison with fmincon.}
  \label{fig:bigexp}
 \end{minipage}
\end{figure}

\section{Conclusion}\label{sec:conclusion}
This paper proposed a fast outer-approximation method for the Minkowski sum of matrix ellipsoids. 
By introducing a parameterized family of bounding ellipsoids, we circumvented the computational bottleneck of conventional LMI characterizations. 
We derived an exact analytical solution for minimizing the sum of squared semi-axes, and developed an efficient majorization-minimization algorithm with closed-form updates for volume minimization.
Theoretical analysis of this algorithm ensures monotonic volume decrease and asymptotic convergence to stationary points. Numerical experiments confirmed our method's superior computational speed and scalability over standard interior-point solvers.
% First, we have introduced a parameterized family of bounding matrix ellipsoids to avoid the computational bottleneck of conventional LMI-based characterizations. 
% Next, we have derived an analytical solution for minimizing the sum of squared semi-axes of the approximating ellipsoid. 
% For volume minimization, we have developed an efficient parameter update algorithm based on the majorization-minimization framework, providing closed-form update rules. 
% We have theoretically proved that the algorithm monotonically decreases the volume and that the parameter sequence asymptotically approaches the set of stationary points. 
% Finally, numerical experiments have demonstrated the superior computational speed and scalability of the proposed method compared to standard interior-point solvers.

\section*{Appendix}
The following proposition stems from the concavity of log-determinant function and constitutes an upper bound via its first-order approximation~\cite[Sec. 3.1]{boyd2004convex}.

\begin{aprp}\label{aprp:logdet}
    Let $X, X_0 \in \mathbb{S}^n_{++}$. Then, it holds that
    \begin{equation}
        \logdet(X) \le \logdet(X_0) + \tr(X_0^{-1}(X - X_0)).
    \end{equation}
\end{aprp}


%%%%%%%%%%%%%%%%%%%%%%%%%%%%%%%%%%%%%%%%%%%%%%%%%%%%%%%%%%%%%%%%%%%%%%%%%%%%%%%%

%%%%%%%%%%%%%%%%%%%%%%%%%%%%%%%%%%%%%%%%%%%%%%%%%%%%%%%%%%%%%%%%%%%%%%%%%%%%%%%%

%%%%%%%%%%%%%%%%%%%%%%%%%%%%%%%%%%%%%%%%%%%%%%%%%%%%%%%%%%%%%%%%%%%%%%%%%%%%%%%%

% \section*{ACKNOWLEDGMENT}

\bibliographystyle{IEEEtran}
\bibliography{IEEEabrv,Control,Robust_DDC,DDC,ellipsoid,mathematics}

@book{boyd1994linear,
  title={Linear Matrix Inequalities in System and Control Theory},
  author={Boyd, Stephen and El Ghaoui, Laurent and Feron, Eric and Balakrishnan, Venkataramanan},
  year={1994},
  publisher={SIAM}
}

@article{Toh1999_SDPT3,
author = {K. C. Toh and M. J. Todd and R. H. Tütüncü},
title = {SDPT3 — A Matlab software package for semidefinite programming, Version 1.3},
journal = {Optimization Methods and Software},
volume = {11},
number = {1-4},
pages = {545--581},
year = {1999},
publisher = {Taylor \& Francis},
doi = {10.1080/10556789908805762},
}

@misc{cvx,
  author       = {CVX Research, Inc.},
  title        = {{CVX}: Matlab Software for Disciplined Convex Programming, version 2.0},
  howpublished = {\url{https://cvxr.com/cvx}},
  year         = 2012
}

@incollection{gb08,
  author    = {M. Grant and S. Boyd},
  title     = {Graph implementations for nonsmooth convex programs},
  booktitle = {Recent Advances in Learning and Control},
  series    = {Lecture Notes in Control and Information Sciences},
  editor    = {V. Blondel and S. Boyd and H. Kimura},
  publisher = {Springer-Verlag Limited},
  pages     = {95--110},
  year      = 2008,}

@STRING{IEEE_J_AC         = "{IEEE} Trans. Automat. Contr."}

@STRING{IEEE_J_SP         = "{IEEE} Trans. Signal Processing"}

@ARTICLE{DDCQMI:Waarde2022_TAC_origin,
  author={van Waarde, Henk J. and Camlibel, M. Kanat and Mesbahi, Mehran},
  journal=IEEE_J_AC, 
  title={From Noisy Data to Feedback Controllers: Nonconservative Design via a Matrix {S}-Lemma}, 
  year={2022},
  volume={67},
  number={1},
  pages={162-175},
  doi={10.1109/TAC.2020.3047577}
}

@article{DDCQMI:BISOFFI2022_Petersen,
title = {Data-driven control via {P}etersen's lemma},
journal = {Automatica},
volume = {145},
number = {110537},
year = {2022},
issn = {0005-1098},
author = {Andrea Bisoffi and Claudio {De Persis} and Pietro Tesi},
}

@article{DDCQMI:BISOFFI2021,
title = {Trade-offs in learning controllers from noisy data},
journal = {Systems \& Control Letters},
volume = {154},
number = {104985},
year = {2021},
issn = {0167-6911},
author = {Andrea Bisoffi and Claudio {De Persis} and Pietro Tesi},
}

@article{DDCQMI:Waarde2023_siam_qmi,
author = {van Waarde, Henk J. and Camlibel, M. Kanat and Eising, Jaap and Trentelman, Harry L.},
title = {Quadratic Matrix Inequalities with Applications to Data-Based Control},
journal = {SIAM Journal on Control and Optimization},
volume = {61},
number = {4},
pages = {2251-2281},
year = {2023},
}

@ARTICLE{DDCQMI:Waarde2024_TAC_AR,
  author={van Waarde, Henk J. and Eising, Jaap and Camlibel, M. Kanat and Trentelman, Harry L.},
  journal=IEEE_J_AC, 
  title={A Behavioral Approach to Data-Driven Control With Noisy Input-Output Data}, 
  year={2024},
  volume={69},
  number={2},
  pages={813-827},
  doi={10.1109/TAC.2023.3275014}
}

@ARTICLE{DDCQMI:Steentjes2022_Cont_Sys_Let_Covariance,
  author={Steentjes, Tom R. V. and Lazar, Mircea and Van den Hof, Paul M. J.},
  journal={IEEE Control Systems Letters}, 
  title={On Data-Driven Control: Informativity of Noisy Input-Output Data With Cross-Covariance Bounds}, 
  year={2022},
  volume={6},
  number={},
  pages={2192-2197},
  doi={10.1109/LCSYS.2021.3139526}
}

@ARTICLE{DDCQMI:BISOFFI2024_CSL,
  author={Bisoffi, Andrea and Li, Lidong and Persis, Claudio De and Monshizadeh, Nima},
  journal={IEEE Control Systems Letters}, 
  title={Controller Synthesis for Input-State Data With Measurement Errors}, 
  year={2024},
  volume={8},
  number={},
  pages={1571-1576},
  doi={10.1109/LCSYS.2024.3402135}
}

@misc{DDCQMI:kaminaga2025datainformativity,
      title={Data Informativity under Data Perturbation}, 
      author={Taira Kaminaga and Hampei Sasahara},
      year={2025},
      eprint={2505.01641},
      archivePrefix={arXiv},
      primaryClass={math.OC},
      url={https://arxiv.org/abs/2505.01641}, 
}

@INPROCEEDINGS{DDCQMI:Kaminaga2025_ACC,
  author={Kaminaga, Taira and Sasahara, Hampei},
  booktitle={2025 American Control Conference (ACC)}, 
  title={Data Informativity for Quadratic Stabilization under Data Perturbation}, 
  year={2025},
  volume={},
  number={},
  pages={},
  doi={}
}

@article{DDCQMI:Lidong2026,
title = {Controller synthesis from noisy-input noisy-output data},
journal = {Automatica},
volume = {183},
pages = {112545},
year = {2026},
issn = {0005-1098},
doi = {https://doi.org/10.1016/j.automatica.2025.112545},
author = {Lidong Li and Andrea Bisoffi and Claudio {De Persis} and Nima Monshizadeh},
}

@ARTICLE{DDCQMI:Hu2025_RDPC,
  author={Hu, Kaijian and Liu, Tao},
  journal=IEEE_J_AC, 
  title={Robust Data-Driven Predictive Control for Unknown Linear Systems with Bounded Disturbances}, 
  year={2025},
  volume={},
  number={},
  pages={1-16},
  doi={10.1109/TAC.2025.3560697}
}

@article{ellip:durieu_multi-input_2001,
	title = {Multi-Input Multi-Output Ellipsoidal State Bounding},
	volume = {111},
	issn = {1573-2878},
	doi = {10.1023/A:1011978200643},
	number = {2},
	journal = {Journal of Optimization Theory and Applications},
	author = {Durieu, C. and Walter, {\'{E}}. and Polyak, B.},
	year = {2001},
	pages = {273--303},
}

@ARTICLE{ellip:kurzhanskiy_TAC_2007,
  author={Kurzhanskiy, Alex A. and Varaiya, Pravin},
  journal=IEEE_J_AC, 
  title={Ellipsoidal Techniques for Reachability Analysis of Discrete-Time Linear Systems}, 
  year={2007},
  volume={52},
  number={1},
  pages={26-38},
  doi={10.1109/TAC.2006.887900}}

@INPROCEEDINGS{ellip:Halder_CDC_2018,
  author={Halder, Abhishek},
  booktitle={2018 IEEE Conference on Decision and Control (CDC)}, 
  title={On the Parameterized Computation of Minimum Volume Outer Ellipsoid of {M}inkowski Sum of Ellipsoids}, 
  year={2018},
  volume={},
  number={},
  pages={4040-4045},
  doi={10.1109/CDC.2018.8619508}}

@article{ellip:FOGEL_Automatica_1982,
title = {On the value of information in system identification—Bounded noise case},
journal = {Automatica},
volume = {18},
number = {2},
pages = {229-238},
year = {1982},
issn = {0005-1098},
doi = {https://doi.org/10.1016/0005-1098(82)90110-8},
author = {Eli Fogel and Y.F. Huang}
}

@article{ellip:yan_closed-form_2015,
	title = {Closed-form characterization of the {Minkowski} sum and difference of two ellipsoids},
	volume = {177},
	issn = {1572-9168},
	doi = {10.1007/s10711-014-9981-3},
	number = {1},
	journal = {Geometriae Dedicata},
	author = {Yan, Yan and Chirikjian, Gregory S.},
	year = {2015},
	pages = {103--128},
}

@InProceedings{ellip:kurzhanski2000,
author="Kurzhanski, Alexander B.
and Varaiya, Pravin",
editor="Lynch, Nancy
and Krogh, Bruce H.",
title="Ellipsoidal Techniques for Reachability Analysis",
booktitle="Hybrid Systems: Computation and Control",
year="2000",
publisher="Springer Berlin Heidelberg",
address="Berlin, Heidelberg",
pages="202--214",
isbn="978-3-540-46430-3"
}

@book{ben2001lectures,
  title={Lectures on Modern Convex Optimization: Analysis, Algorithms, and Engineering Applications},
  author={Ben-Tal, Aharon and Nemirovski, Arkadi},
  year={2001},
  publisher={SIAM}
}

@book{boyd2004convex,
  title={Convex Optimization},
  author={Boyd, Stephen P and Vandenberghe, Lieven},
  year={2004},
  publisher={Cambridge University Press}
}

@book{math:luenberger2021linearnonlinear,
  title={Linear and Nonlinear Programming},
  author={Luenberger, David G. and Ye, Yinyu},
  year={2021},
  edition={5},
  publisher={Springer},
  series={International Series in Operations Research \& Management Science},
  volume={228},
  doi={10.1007/978-3-030-85450-8},
  isbn={978-3-030-85450-8},
}

@ARTICLE{math:sum2017MMalgo,
  author={Sun, Ying and Babu, Prabhu and Palomar, Daniel P.},
  journal=IEEE_J_SP , 
  title={Majorization-Minimization Algorithms in Signal Processing, Communications, and Machine Learning}, 
  year={2017},
  volume={65},
  number={3},
  pages={794-816},
  doi={10.1109/TSP.2016.2601299}}

\end{document}